\documentclass[leqno, 11pt]{amsart}

\usepackage[top=30truemm,bottom=30truemm,left=25truemm,right=25truemm]{geometry}
\usepackage{graphicx}
\usepackage{comment}

\usepackage{amsmath, amssymb, amsthm, mathdots, mathtools}
\usepackage{mathrsfs, latexsym, bm, stmaryrd}
\usepackage{braket}
\usepackage[all,color]{xy}

\usepackage{ytableau}
\ytableausetup{boxsize=1.4em, centertableaux}
\usepackage[breakable, skins]{tcolorbox}
\usepackage{tikz}
\usetikzlibrary{positioning, shapes.misc, decorations.pathreplacing, fit}
\usepackage{float}

\allowdisplaybreaks[2]
\newcommand{\dd}{\mathrm{d}}

\usepackage[nameinlink]{cleveref}

\theoremstyle{plain}
\newtheorem{thm}{Theorem}[section]
\newtheorem{lem}[thm]{Lemma}
\newtheorem{prop}[thm]{Proposition}

\theoremstyle{definition}

\newtheorem{rem}[thm]{Remark}

\crefname{thm}{Theorem}{Theorems}
\crefname{lem}{Lemma}{Lemmas}
\crefname{prop}{Proposition}{Propositions}
\crefname{cor}{Corollary}{Corollaries}
\crefname{fact}{Fact}{Facts}
\crefname{conj}{Conjecture}{Conjectures}
\crefname{dfn}{Definition}{Definitions}
\crefname{rem}{Remark}{Remarks}

\mathtoolsset{showonlyrefs}

\newtcolorbox{mybox}[1]{
  colframe=black,
  colbacktitle=gray!50,
  coltitle=black,
  colback=gray!10,
  title={#1},
  fonttitle=\bfseries,
  arc=3mm,
  boxrule=1pt,
  titlerule=0.5pt,
  breakable
}

\title{Applications of a formula of Maesaka-Seki-Watanabe type 
for multiple harmonic $q$-sums}
\date{}
\author{Yoshihiro Takeyama}
\address{Department of Mathematics, Institute of Pure and Applied Sciences, University of Tsukuba, Ibaraki 305-8571, Japan}
\email{takeyama@math.tsukuba.ac.jp} 

\author{Yuto Tsuruta}
\address{Mathematical Institute, Tohoku University, Sendai 980-8578, Japan}
\email{tsuruta.yuuto.q7@dc.tohoku.ac.jp}

\begin{document}

\begin{abstract}
    Maesaka, Seki and Watanabe proved a formula for multiple harmonic sums.
    Yamamoto generalized it to Schur-type multiple harmonic sums,
    and the second author proved a $q$-analogue of this generalization.
    In this paper, we give two applications of the $q$-analogue formula.
    The first is an alternative proof of the duality of
    a $q$-analogue of multiple zeta values.
    The second is a proof of an identity for a $q$-analogue of
    the Kawashima function.
\end{abstract}

\maketitle

\section{Introduction}

We call a tuple of positive integers an \textit{index}.
An index $(k_{1}, \ldots , k_{r})$ is said to be \textit{admissible} if
$k_{r}\ge 2$.
For an admissible index $\bm{k}=(k_{1}, \ldots , k_{r})$,
\textit{multiple zeta value} (MZV for short)
$\zeta(\bm{k})$ is defined by
\begin{align}
    \zeta(\bm{k})=
    \sum_{0<m_{1}<\cdots<m_{r}}\frac{1}{m_{1}^{k_{1}}\cdots m_{r}^{k_{r}}}.
\end{align}
In \cite{MSWoriginal},
Maesaka, Seki and Watanabe obtained an interesting formula,
which we call the MSW formula in this paper,
for the truncated sum
\begin{align*}
    \zeta_{<N}(\bm{k})=
    \sum_{0<m_{1}<\cdots<m_{r}<N}
    \frac{1}{m_{1}^{k_{1}}\cdots m_{r}^{k_{r}}}.
\end{align*}

\begin{thm}{\cite[Theorem 1.3]{MSWoriginal}}
    \label{Thm: MSW formula}
    For an index $\bm{k}=(k_{1}, \ldots , k_{r})$ and
    a positive integer $N$, we set
    \begin{align}
        \zeta_{<N}^{\flat}(\bm{k})=
        \sum
        \prod_{j=1}^{r}
        \left(
        \frac{1}{N-n_{j,1}}
        \prod_{l=2}^{k_{j}}\frac{1}{n_{j,l}}
        \right),
    \end{align}
    where the sum is taken over all integers
    $n_{j, l} \, (1\le j \le r, \, 1\le l \le k_{j})$ satisfying
    \begin{align}
        0<n_{j,1}\leq\cdots\leq n_{j,k_{j}}<N \quad (1\le j \le r),
        \qquad
        n_{j,k_{j}}<n_{(j+1),1}\quad (1\leq j<r).
        \label{eq:sum-condition-MSW}
    \end{align}
    Then, for any index $\bm{k}$ and any positive integer $N$,
    it holds that
    \begin{align}
        \zeta_{<N}(\bm{k})=\zeta_{<N}^{\flat}(\bm{k}).
        \label{eq:MSW}
    \end{align}
\end{thm}

For example, $\zeta^{\flat}_{<N}(1,2,3)$ is written as
\begin{align}
    \zeta^{\flat}_{<N}(1,2,3)=\sum_{0<n_{1}<n_{2}\leq n_{3}<n_{4}\leq n_{5}\leq n_{6}<N}
    \frac{1}{(N-n_{1})(N-n_{2})n_{3}(N-n_{4})n_{5}n_{6}}.
\end{align}
In the limit as $N\to \infty$,
the sum converges to the multiple integral
\begin{align*}
    \idotsint\limits_{0<t_{1}<t_{2}<t_{3}<t_{4}<t_{5}<t_{6}<1}
    \frac{\dd t_{1}}{1-t_{1}}\frac{\dd t_{2}\dd t_{3}}{(1-t_{2})t_{3}}\frac{\dd t_{4}\dd t_{5}\dd t_{6}}{(1-t_{4})t_{5}t_{6}},
\end{align*}
which is the iterated integral representation of
the MZV $\zeta(1, 2, 3)$.
In this sense, the MSW formula \eqref{eq:MSW} gives
a finite discretization of the iterated integral representation of MZVs.

In \cite{MSWyamamoto},
Yamamoto generalized the MSW formula \eqref{eq:MSW} to
Schur-type multiple harmonic sums,
which may be viewed as a finite discretization
of the integral formula for Schur multiple zeta values
established by Hirose, Murahara and Onozuka \cite{HMOSchurInt}.
It contains the star-version of the MSW formula:

\begin{thm}{\cite[Theorem 2.1]{MSWyamamoto}}
    For an index $\bm{k}=(k_{1}, \ldots , k_{r})$
    and a positive integer $N$, we set
    \begin{align*}
        \zeta_{<N}^{\star}(\bm{k})=
        \sum_{0<m_{1}\le \cdots \le m_{r}<N}
        \frac{1}{m_{1}^{k_{1}} \cdots m_{r}^{k_{r}}}
    \end{align*}
    and
    \begin{align*}
        \zeta_{<N}^{\star \flat}(\bm{k})=
        \sum
        \prod_{j=1}^{r}
        \left(
        \frac{1}{N-n_{j,1}}
        \prod_{l=2}^{k_{j}}\frac{1}{n_{j,l}}
        \right),
    \end{align*}
    where the sum is taken over all integers
    $n_{j, l} \, (1\le j \le r, \, 1\le l \le k_{j})$
    satisfying
    \begin{align}
        0<n_{j,1}\leq\cdots\leq n_{j,k_{j}}<N \quad (1\le j \le r),
        \qquad
        n_{j,k_{j}} \ge n_{j-1,1} \quad (1<j\le r).
        \label{eq:sum-condition-star-MSW}
    \end{align}
    Then, for any index $\bm{k}$ and any positive integer $N$, it holds that
    \begin{align}
        \zeta_{<N}^{\star}(\bm{k})=
        \zeta_{<N}^{\star \flat}(\bm{k}).
        \label{eq:MSW-star}
    \end{align}
\end{thm}

In \cite{Myarticle},
the second author obtained
a $q$-analogue of Yamamoto's generalization.
It contains a $q$-analogue of
\eqref{eq:MSW} and \eqref{eq:MSW-star}
given as follows.

We assume throughout this paper that
$q$ is a constant satisfying
\begin{align*}
    0<q<1.
\end{align*}
For an integer $m$, the $q$-integer $[m]$ is defined by
$[m]=(1-q^{m})/(1-q)$.

For an index $\bm{k}=(k_{1}, \ldots , k_{r})$, we set
\begin{align}
    \zeta_{<N}^{q}(\bm{k})=
    \sum_{0<m_{1}<\cdots<m_{r}<N}
    \frac{q^{(k_{1}-1)m_{1}+\cdots+(k_{r}-1)m_{r}}}{[m_{1}]^{k_{1}}\cdots [m_{r}]^{k_{r}}}
    \label{eq:q-flat-sum}
\end{align}
and define its star-version
$\zeta_{<N}^{\star, q}(\bm{k})$
as \eqref{eq:q-flat-sum} with the summation range
replaced by $0<m_{1}\le \cdots\le m_{r}<N$.
Note that $\zeta_{<N}^{q}(\bm{k})$ is a truncated sum of
the Bradley-Zhao model of
a $q$-analogue of MZV ($q$MZV for short) defined by
\begin{align}
    \zeta^{q}(\bm{k})=
    \sum_{0<m_{1}<\cdots<m_{r}}\frac{q^{(k_{1}-1)m_{1}+\cdots+(k_{r}-1)m_{r}}}{[m_{1}]^{k_{1}}\cdots [m_{r}]^{k_{r}}}
    \label{eq:qMZV}
\end{align}
for an admissible index $\bm{k}=(k_{1}, \ldots, k_{r})$.

\begin{thm}{\cite[Special cases of Theorem 1.2]{Myarticle}}
    \label{Thm: q-MSW formula}
    For an index $\bm{k}=(k_{1}, \ldots , k_{r})$ and
    a positive integer $N$, we set
    \begin{align}
        \zeta_{<N}^{q\flat}(\bm{k})=
        \sum\prod_{j=1}^{r}
        \left(
        \frac{1}{[N-n_{j,1}]}
        \prod_{l=2}^{k_{j}}
        \frac{q^{n_{j,l}}}{[n_{j,l}]}
        \right),
        \label{eq:q-flat}
    \end{align}
    where the sum is taken over all integers
    $n_{j, l} \, (1\le j \le r, \, 1\le l \le k_{j})$ satisfying
    \eqref{eq:sum-condition-MSW}.
    We also define its star-version $\zeta_{<N}^{\star, q\flat}(\bm{k})$
    as \eqref{eq:q-flat} with the summation range replaced by
    \eqref{eq:sum-condition-star-MSW}.
    Then, for any index $\bm{k}$ and
    any positive integer $N$, we have
    \begin{align}
        \zeta^{q}_{<N}(\bm{k})
         & =\zeta^{q\flat}_{<N}(\bm{k}),
        \label{eq:qMSW}                         \\
        \zeta^{\star, q}_{<N}(\bm{k})
         & =\zeta^{\star, q\flat}_{<N}(\bm{k}).
        \label{eq:star-qMSW}
    \end{align}
\end{thm}

The purpose of this paper is to discuss
two applications of the formulas
\eqref{eq:qMSW} and \eqref{eq:star-qMSW}.
The first is an alternative proof for the duality
of $q$MZVs \eqref{eq:qMZV}
using the formula \eqref{eq:qMSW}.
The second is a proof of an identity
of a $q$-analogue of the Kawashima function
using the formula \eqref{eq:star-qMSW}.

In Section \ref{sec:duality},
we present a new proof of the duality of $q$MZVs.

Let $\bm{k}$ be an admissible index.
There uniquely exist positive integers $s$ and
$a_{1}, \ldots , a_{s}, b_{1}, \ldots , b_{s}$
such that
\begin{align}
    \bm{k}=(\underbrace{1,\ldots,1}_{a_{1}-1},b_{1}+1,\ldots,\underbrace{1,\ldots,1}_{a_{s}-1},b_{s}\color{black}+1).
    \label{eq:def-ab}
\end{align}
Then the dual index $\bm{k}^{\dagger}$ of $\bm{k}$ is defined by
\begin{align*}
    \bm{k}^{\dagger}=(\underbrace{1,\ldots,1}_{b_{s}-1},a_{s}+1,\ldots,\underbrace{1,\ldots,1}_{b_{1}-1},a_{1}+1).
\end{align*}

\begin{thm}[Duality of $q$MZVs, \cite{Bradley1}]\label{Thm: Ordinary duality}
    For any admissible index $\bm{k}$,
    it holds that $\zeta^{q}(\bm{k})=\zeta^{q}(\bm{k}^{\dagger})$.
\end{thm}

By taking the limit as $q\to 1$, we recover the duality
$\zeta(\bm{k})=\zeta(\bm{k}^{\dagger})$ of MZVs.

The paper \cite{MSWoriginal} provides
a proof of the duality of MZVs by means of the MSW formula.
It is proved that,
for any admissible index $\bm{k}$, there exists $J>0$ such that
\begin{align*}
    \zeta_{<N}(\bm{k})-\zeta_{<N}(\bm{k}^{\dagger})=
    N^{-1}O((\log{N})^{J}) \qquad  (N \to \infty).
\end{align*}
By taking the limit as $N\to \infty$,
we obtain the duality
$\zeta(\bm{k})=\zeta(\bm{k}^{\dagger})$.

In contrast to the MZV case,
our proof of the duality of $q$MZVs
proceeds as follows.
Using the formula \eqref{eq:qMSW},
one can obtain,
for any admissible index $\bm{k}$,
a finite sum $\phi_{<N}^{q}(\bm{k})$
such that
\begin{align*}
    \zeta_{<N}^{q}(\bm{k})=\zeta_{<N}^{q\flat}(\bm{k})=\phi_{<N}^{q}(\bm{k})+q^{N}O(N^{J})
\end{align*}
with some $J>0$ and
\begin{align}
    \phi_{<N}^{q}(\bm{k}) \to \zeta^{q}(\bm{k}^{\dagger})
    \label{eq:phi-zeta}
\end{align}
as $N\to \infty$.
The duality is then derived from this.
To show \eqref{eq:phi-zeta},
we employ the same technique as was used to prove
the resummation identity given in \cite{Takeyama2}.

In Section \ref{sec:Kawashima},
we prove an identity of a $q$-analogue of the Kawashima function.

In \cite{Kawashima1}, Kawashima introduced a function
$F_{\bm{k}}(z)$ for an index $\bm{k}$, which we call the Kawashima function.
The function $F_{\bm{k}}(z)$ interpolates the truncated sum
$\zeta_{<N}(\bm{k})$ as
$F_{\bm{k}}(N-1)=\zeta_{<N}(\bm{k})$
for any positive integer $N$.
In \cite{Kawashima2}, Kawashima obtained an alternative expression $G_{\bm{k}}(z)$
of the Kawashima function such that
\begin{align}
    F_{\bm{k}}(z)=G_{\overleftarrow{\bm{k}}}(z)
    \label{eq:Kawashima-identity}
\end{align}
in the region $\mathop{\mathrm{Re}}{z}>-1$,
where $\overleftarrow{\bm{k}}:=(k_{r}, k_{r-1}, \ldots , k_{1})$ for
$\bm{k}=(k_{1}, \ldots , k_{r})$.
The key to the proof of \eqref{eq:Kawashima-identity}
is the identity
\begin{align}
    F_{\bm{k}}(N-1)=G_{\overleftarrow{\bm{k}}}(N-1)
    \label{eq:key-identity}
\end{align}
for any positive integer $N$.
In \cite{MSWyamamoto}, Yamamoto pointed out that
the identity \eqref{eq:key-identity} is nothing but the star-version
\eqref{eq:MSW-star} of the MSW formula.

In \cite{Takeyama1}, the first author defined a $q$-analogue
$F_{\bm{k}}^{q}(z)$ of the Kawashima function.
% and proved a family of quadratic relations among $q$MZVs.
In Section \ref{sec:Kawashima}, we introduce a $q$-analogue $G_{\bm{k}}^{q}(z)$
of the function $G_{\bm{k}}(z)$
and show the identity $F_{\bm{k}}^{q}(z)=G_{\overleftarrow{\bm{k}}}^{q}(z)$
on the disk $|z|<q^{-1}$,
which follows from the formula \eqref{eq:star-qMSW} and
the identity theorem for analytic functions.

Throughout this paper we set
$|\bm{k}|=\sum_{a=1}^{r}k_{a}$
for an index $\bm{k}=(k_{1}, \ldots , k_{r})$.

\section*{Acknowledgements}
This work was supported by JSPS KAKENHI Grant Number JP22K03243.

\section{A new proof of duality of $q$MZVs}
\label{sec:duality}

In this section, we give an alternative proof of
Theorem \ref{Thm: Ordinary duality}.

\begin{prop}\label{Prop: asym exp of qMZV}
    For an index $\bm{k}=(k_{1}, \ldots , k_{r})$ and a positive integer $N$,
    we set
    \begin{align*}
        \phi_{<N}^{q}(\bm{k})=(1-q)^{r}
        \sum \prod_{j=1}^{r}
        \left(
        \prod_{l=2}^{k_{j}}\frac{q^{n_{j,l}}}{[n_{j,l}]}
        \right),
    \end{align*}
    where the sum is taken over all integers $n_{j, l} \, (1\le j \le r, \, 1\le l \le k_{j})$
    satisfying \eqref{eq:sum-condition-MSW}.
    Then, for any admissible index $\bm{k}$, it holds that
    \begin{align}\label{eq: asymptotic1}
        \zeta_{<N}^{q\flat}(\bm{k})=\phi_{<N}^{q}(\bm{k})+q^{N}O(N^{J})
        \qquad (N \to \infty)
    \end{align}
    for some $J>0$ independent of $N$.
\end{prop}

\begin{proof}
    Suppose that $\bm{k}=(k_{1}, \ldots , k_{r})$ is admissible.
    Note that $k_{r}\ge 2$.
    Since $0<q<1$, we have $1/[m]\le 1$ for any positive integer $m$.
    Therefore, if $0<n_{j,1}\le n_{r, k_{r}}<N$, it holds that
    \begin{align}\label{eq: key tool}
        0<\frac{1}{[N-n_{j,1}]}\frac{q^{n_{r, k_{r}}}}{[n_{r, k_{r}}]}=
        \left(1-q+\frac{q^{N-n_{j,1}}}{[N-n_{j,1}]}\right)
        \frac{q^{n_{r, k_{r}}}}{[n_{r, k_{r}}]}
        \le (1-q)\frac{q^{n_{r, k_{r}}}}{[n_{r, k_{r}}]}+q^{N}.
    \end{align}
    Applying \eqref{eq: key tool} to $\zeta^{q\flat}_{<N}(\bm{k})$
    for $1\le j\le r$, we obtain \eqref{eq: asymptotic1}.
\end{proof}

\begin{lem}
    Suppose that $a, b\ge 1$ and $c, d\ge 0$.
    It holds that
    \begin{align}
        \label{eq:duality-lem}
        (1-q)^{a}\sum_{d<n_{1}\le \cdots \le n_{b}}
        \binom{n_{1}-d}{a}
        \left( \prod_{j=1}^{b-1}\frac{q^{n_{j}}}{[n_{j}]}
        \right)
        \frac{q^{n_{b}}}{[n_{b}]}q^{cn_{b}}=
        \sum_{c<m_{1}<\cdots <m_{b}}
        \left(
        \prod_{j=1}^{b-1}\frac{1}{[m_{j}]}
        \right)
        \frac{q^{am_{b}}}{[m_{b}]^{a+1}}
        q^{dm_{b}}.
    \end{align}
\end{lem}

\begin{proof}
    We denote the left-hand side by $I$.
    We have
    \begin{align*}
        I=(1-q)^{a+b}\sum_{d\le n_{1}\le \cdots \le n_{b}}
        \sum_{l_{1}, \ldots, l_{b}\ge 1}
        \binom{n_{1}-d}{a}
        q^{\sum_{j=1}^{b}n_{j}l_{j}+cn_{b}}
    \end{align*}
    because $\binom{n_{1}-d}{a}=0$ if $n_{1}=d$.
    We perform the change of summation variables
    $(n_{1}, \ldots , n_{b}, l_{1}, \ldots , l_{b})
        \mapsto
        (s_{1}, \ldots , s_{b}, m_{1}, \ldots , m_{b})$
    defined by
    \begin{align*}
        s_{1}=n_{1}-d, \qquad s_{j}=n_{j}-n_{j-1} \quad (2\le j \le b)
    \end{align*}
    and
    \begin{align*}
        m_{j}=c+\sum_{i=b-j+1}^{b}l_{i} \quad (1\le j \le b).
    \end{align*}
    The sum is then taken over integers $s_{j}$ and $m_{j} \, (1\le j \le b)$
    satisfying $s_{1}, \ldots , s_{b}\ge 0$ and
    $c<m_{1}<\cdots <m_{b}$.
    It holds that
    \begin{align*}
        \sum_{j=1}^{b}n_{j}l_{j}+cn_{b}=\sum_{j=1}^{b}s_{j}m_{b-j+1}+dm_{b}.
    \end{align*}
    We take the sum over $s_{1}, \ldots , s_{b}\ge 0$.
    By using
    \begin{align*}
        \sum_{s\ge 0}\binom{s}{a}x^{s}=\frac{x^{a}}{(1-x)^{a+1}} \qquad (|x|<1),
    \end{align*}
    we see that
    \begin{align*}
        I=(1-q)^{a+b}\sum_{c<m_{1}<\cdots <m_{b}}
        \left(\prod_{j=1}^{b-1}\frac{1}{1-q^{m_{j}}}\right)
        \frac{q^{am_{b}}}{(1-q^{m_{b}})^{a+1}}\,
        q^{dm_{b}},
    \end{align*}
    which is equal to the right-hand side of \eqref{eq:duality-lem}.
\end{proof}

Now, we are ready to prove Theorem \ref{Thm: Ordinary duality}.

\begin{proof}[Proof of Theorem \ref{Thm: Ordinary duality}]
    Let $\bm{k}$ be an admissible index.
    We define positive integers $a_{j}, b_{j} \, (1\le j \le s)$
    by \eqref{eq:def-ab}.
        {}From the definition of $\phi_{<N}^{q}(\bm{k})$,
    we see that
    \begin{align}
        \label{eq:phi-limit}
        \lim_{N\to \infty}\phi_{<N}^{q}(\bm{k})=(1-q)^{\sum_{j=1}^{s}a_{j}}
        \sum
        \prod_{j=1}^{s}\left(
        \binom{n_{j, 1}-n_{j-1,b_{j-1}}}{a_{j}}
        \prod_{l=1}^{b_{j}}\frac{q^{n_{j,l}}}{[n_{j,l}]}
        \right),
    \end{align}
    where $n_{0, b_{0}}=0$ and the sum is taken over
    all integers $n_{j, l} \, (1\le j\le s, \, 1\le l\le b_{j})$
    satisfying
    \begin{align*}
        0<n_{1, 1}\le \cdots \le n_{1, b_{1}}<
        n_{2,1}\le \cdots \le n_{2, b_{2}}<\cdots
        <n_{s, 1}\le \cdots \le n_{s, b_{s}}.
    \end{align*}
    By using \eqref{eq:duality-lem} repeatedly,
    we see that
    the right-hand side of \eqref{eq:phi-limit} is equal to
    $\zeta^{q}(\bm{k}^{\dagger})$.
\end{proof}

\section{An identity of a $q$-analogue of Kawashima function}
\label{sec:Kawashima}

For an index $\bm{k}$, we define its Hoffman dual $\bm{k}^{\vee}$ as follows.
We write $\bm{k}$ in the form
$\bm{k}=(1 \square 1 \square \cdots \square 1)$,
where $\square$ is either $+$ (plus symbol) or $,$ (comma).
Then $\bm{k}^{\vee}$ is the index obtained by changing $+$ to $,$ and
vice versa.
For example, if $\bm{k}=(2, 1, 3, 2)=(1+1, 1, 1+1+1, 1+1)$,
then $\bm{k}^{\vee}=(1, 1+1+1,1,1+1,1)=(1,3,1,2,1)$.

Let $\bm{k}$ be an index.
We set $\bm{k}^{\vee}=(k_{1}', \ldots , k_{s}')$
and define the function $F_{\bm{k}}^{q}(z)$, which is a $q$-analogue
of the Kawashima function, by
\begin{align}
    F_{\bm{k}}^{q}(z)=-\sum_{0<m_{1}\le \cdots \le m_{s}}
    \left(
    \prod_{a=1}^{s-1}\frac{q^{(k_{a}'-1)m_{a}}}{[m_{a}]^{k_{a}'}}
    \right)
    \frac{q^{k_{s}'m_{s}}}{[m_{s}]^{k_{s}'}}
    \prod_{j=1}^{m_{s}}\frac{z-q^{j-1}}{1-q^{j}}.
    \label{eq:q-Kawashima-F}
\end{align}

\begin{prop}
    The infinite sum in \eqref{eq:q-Kawashima-F} converges absolutely and
    defines an analytic function in the region $|z|<q^{-1}$.
\end{prop}

\begin{proof}
    Fix a constant $c$ with $0<c<q^{-1}$.
    Suppose that $|z|\le c$.
    Set $d=q(1-qc)/(1+q)$.
    Note that $0<d<1$.
    We take a positive integer $M$ such that
    $q^{M}<d$.
    Then, if $j>M$, it holds that
    \begin{align*}
        \left|\frac{z-q^{j-1}}{1-q^{j}}\right|
         & \le
        \left|\frac{z-q^{j-1}}{1-q^{j}}-z\right|+|z|=
        \frac{q^{j-1}}{1-q^{j}}|1-zq|+|z|
        \\
         & \le
        \frac{q^{M}}{1-q^{M}}(1+c)+c\le \frac{c+d}{1-d}.
    \end{align*}
    For $1\le j\le M$, we have
    \begin{align*}
        \left|\frac{z-q^{j-1}}{1-q^{j}}\right| \le
        \frac{c+1}{1-q}.
    \end{align*}
    For any positive integer $m$, we have $1/[m]\le 1$.
    Therefore, if $0<m_{1}\le \cdots \le m_{s}$ and $m_{s}>M$,
    it holds that
    \begin{align*}
        \left|
        \left(
        \prod_{a=1}^{s-1}\frac{q^{(k_{a}'-1)m_{a}}}{[m_{a}]^{k_{a}'}}
        \right)
        \frac{q^{k_{s}'m_{s}}}{[m_{s}]^{k_{s}'}}
        \prod_{j=1}^{m_{s}}\frac{z-q^{j-1}}{1-q^{j}}
        \right|\le
        q^{m_{s}} \left(\frac{c+1}{1-q}\right)^{M}
        \left(\frac{c+d}{1-d}\right)^{m_{s}-M}
    \end{align*}
    because $k_{s}'\ge 1$.
        {}From the definition of $d$ and $0<c<q^{-1}$, we see that
    \begin{align*}
        \frac{c+d}{1-d}=(1+q)\frac{1+c}{1+q^{2}c}-1<q^{-1}.
    \end{align*}
    Hence, we can apply the Weierstrass $M$-test to
    the series in \eqref{eq:q-Kawashima-F}.
\end{proof}

The function $F_{\bm{k}}^{q}(z)$ has the following property.

\begin{prop}{\cite{Takeyama1}}\label{prop:F-value}
    For any index $\bm{k}=(k_{1}, \ldots , k_{r})$ and
    any positive integer $N$, it holds that
    \begin{align*}
        F_{\bm{k}}^{q}(q^{N-1})=\sum_{0<m_{1}\le \cdots \le m_{r}<N}
        \frac{q^{m_{1}+\cdots +m_{r}}}{[m_{1}]^{k_{1}} \cdots [m_{r}]^{k_{r}}}.
    \end{align*}
\end{prop}

For an integer $m$, we set
\begin{align*}
    [m+\log_{q}{z}]=\frac{1-zq^{m}}{1-q}.
\end{align*}
Note that, if $z=q^{N}$ for an integer $N$,
we have $[m+\log_{q}{z}]=[m+N]$.

For an index $\bm{k}=(k_{1}, \ldots , k_{r})$,
we define the function $G_{\bm{k}}^{q}(z)$ by
\begin{align}
    \label{eq:q-Kawashima-G}
    G^{q}_{\bm{k}}(z)=\sum
     &
    \left\{
    \prod_{j=1}^{r-1}
    \left(
    \prod_{l=1}^{k_{a}-1}\frac{1}{[m_{j, l}+\log_{q}{z}]} \,
    \right)
    \frac{1}{[m_{j, k_{j}}]}
    \right\}
    \\ \nonumber
     & \times
    \left(
    \prod_{l=1}^{k_{r}-1}\frac{1}{[m_{r, l}+\log_{q}{z}]}
    \right)
    \left(
    \frac{q^{m_{r,k_{r}}}}{[m_{r,k_{r}}]}-\frac{q^{m_{r,k_{r}}}z}{[m_{r,k_{r}}+\log_{q}{z}]}
    \right),
\end{align}
where the sum is over the region
\begin{align*}
    0<m_{1,1}\leq\cdots\leq m_{1,k_{1}}<m_{2,1}\leq\cdots\leq m_{2,k_{2}}< \cdots
    <m_{r,1}\leq\cdots\leq m_{r,k{r}}.
\end{align*}

\begin{prop}
    The infinite sum in \eqref{eq:q-Kawashima-G} converges absolutely and
    defines an analytic function in the region $|z|<q^{-1}$.
\end{prop}

\begin{proof}
    Let $c$ be a constant satisfying $0<c<q^{-1}$.
    Suppose that $|z|\le c$.
    For any positive integer $m$, it holds that $1/[m]\le 1$ and
    \begin{align*}
        \left|[m+\log_{q}{z}]\right|\ge \frac{1-cq}{1-q}>0.
    \end{align*}
    Hence, we have
    \begin{align*}
        \left| \frac{q^{m}}{[m]}-\frac{q^{m}z}{[m+\log_{q}{z}]} \right|
        \le q^{m}\left(1+\frac{c(1-q)}{1-cq}\right).
    \end{align*}
    Therefore, the Weierstrass $M$-test applies to the series
    in \eqref{eq:q-Kawashima-G}.
\end{proof}

We prove the following lemma,
which will be needed to evaluate $G_{\bm{k}}(z)$ at $z=q^{N-1}$
for a positive integer $N$.

\begin{lem}
    Let $N$ be a positive integer.
    For any positive integers $m', n', n''$ satisfying $n'\le n''\le N-1$,
    the following identities hold.
    \begin{align}
        \label{eq: G-flat(1)}
        \frac{1}{[m^{\prime}+N-1]} &
        \sum_{m=m^{\prime}}^{\infty}
        \left( \frac{q^{m+N-1-n^{\prime\prime}}}{[m+N-1-n^{\prime\prime}]}-
        \frac{q^{m+N-n^{\prime}}}{[m+N-n^{\prime}]} \right)
        \\
                                   & =\sum_{n=n^{\prime}}^{n^{\prime\prime}}
        \left( \frac{q^{m^{\prime}+N-1-n}}{[m^{\prime}+N-1-n]}-
        \frac{q^{m^{\prime}+N-1}}{[m^{\prime}+N-1]} \right)
        \frac{1}{[n]},
        \nonumber
        \\
        \label{eq: G-flat(2)}
        \frac{1}{[m^{\prime}]}     &
        \sum_{m=m^{\prime}+1}^{\infty}
        \left( \frac{q^{m+N-1-n^{\prime\prime}}}{[m+N-1-n^{\prime\prime}]}-
        \frac{q^{m+N-n^{\prime}}}{[m+N-n^{\prime}]} \right)
        \\
                                   & =\sum_{n=n^{\prime}}^{n^{\prime\prime}}
        \left( \frac{q^{m^{\prime}}}{[m^{\prime}]}-
        \frac{q^{m^{\prime}+N-n}}{[m^{\prime}+N-n]} \right)
        \frac{q^{N-n}}{[N-n]}.
        \nonumber
    \end{align}
\end{lem}

\begin{proof}
    Here we prove \eqref{eq: G-flat(1)}.
    The proof of \eqref{eq: G-flat(2)} is similar.
    Since the sum $\sum_{m\ge 1}q^{m}/[m]$ converges and $n'\le n''$,
    we see that the left-hand side is equal to
    \begin{align*}
        \frac{1}{[m'+N-1]}\sum_{n=n'}^{n''}\frac{q^{m'+N-n-1}}{[m'+N-n-1]}.
    \end{align*}
    Using the partial fraction decomposition
    \begin{align*}
        \frac{q^{m'+N-n-1}}{[m'+N-1][m'+N-n-1]}=
        \left( \frac{q^{m^{\prime}+N-1-n}}{[m^{\prime}+N-1-n]}-
        \frac{q^{m^{\prime}+N-1}}{[m^{\prime}+N-1]} \right)
        \frac{1}{[n]},
    \end{align*}
    we obtain \eqref{eq: G-flat(1)}.
\end{proof}

\begin{prop}\label{prop:G-value}
    For any index $\bm{k}=(k_{1}, \ldots , k_{r})$ and
    any positive integer $N$, we have
    \begin{align}
        G^{q}_{\overset{\leftarrow}{\bm{k}}}(q^{N-1})=
        \sum\prod_{j=1}^{r}
        \left(
        \frac{q^{N-n_{j,1}}}{[N-n_{j,1}]}
        \prod_{l=2}^{k_{j}}
        \frac{1}{[n_{j,l}]}
        \right),
        \label{eq:G-value}
    \end{align}
    where the sum is taken over all integers
    $n_{j, l} \, (1\le j \le r, \, 1\le l \le k_{j})$ satisfying
    \eqref{eq:sum-condition-star-MSW}.
\end{prop}

\begin{proof}

    We prove the case $\bm{k}=(2,3)$ as an illustration.
    The general case follows similarly.
    We have
    \begin{align}
        G^{q}_{(3,2)}(q^{N-1})=\sum_{\substack{
                                       0<m_{1}\leq m_{2}\leq m_{3}\\ <m_{4}\leq m_{5}}}
         & \frac{1}{[m_{1}+N-1][m_{2}+N-1][m_{3}][m_{4}+N-1]}
        \\
         & \times\left(
        \frac{q^{m_{5}}}{[m_{5}]}-\frac{q^{m_{5}+N-1}}{[m_{5}+N-1]} \right).
    \end{align}
    Applying \eqref{eq: G-flat(1)} to the sum over $m_{5}$, we obtain

    \begin{align}\label{eq: G-flat1}
        G^{q}_{(3,2)}(q^{N-1})=
        \sum_{0<n_{5}<N}\frac{1}{[n_{5}]}\sum_{0<m_{1}\leq m_{2}\leq m_{3}<m_{4}}
         &
        \frac{1}{[m_{1}+N-1][m_{2}+N-1][m_{3}]}
        \\
         & \times\left(
        \frac{q^{m_{4}+N-1-n_{5}}}{[m_{4}+N-1-n_{5}]}-\frac{q^{m_{4}+N-1}}{[m_{4}+N-1]}
        \right).
    \end{align}
    Using \eqref{eq: G-flat(2)} to the sum over $m_{4}$, we get

    \begin{align}\label{eq: G-flat2}
        G^{q}_{(3,2)}(q^{N-1})=
         & \sum_{0<n_{4}\leq n_{5}<N}
        \frac{q^{N-n_{4}}}{[N-n_{4}][n_{5}]} \\
         & \qquad {}\times
        \sum_{0<m_{1}\leq m_{2}\leq m_{3}}
        \frac{1}{[m_{1}+N-1][m_{2}+N-1]}
        \left( \frac{q^{m_{3}}}{[m_{3}]}-\frac{q^{m_{3}+N-n_{4}}}{[m_{3}+N-n_{4}]} \right).
    \end{align}
    %    By reprating the above steps, we have
    Applying \eqref{eq: G-flat(1)} twice and then \eqref{eq: G-flat(2)},
    we see that the right-hand side is equal to
    \begin{align}
         &
        \sum_{\substack{0<n_{3}<N, \, 0<n_{4}\le n_{5}<N
                \\ n_{3}\ge n_{4}}}
        \frac{1}{[n_{3}]}\frac{q^{N-n_{4}}}{[N-n_{4}][n_{5}]}
        \sum_{0<m_{1}\leq m_{2}}\frac{1}{[m_{1}+N-1]}\left(
        \frac{q^{m_{2}+N-1-n_{3}}}{[m_{2}+N-1-n_{3}]}-\frac{q^{m_{2}+N-1}}{[m_{2}+N-1]}
        \right)                                                        \\
         & =\sum_{\substack{
                    0<n_{2}\leq n_{3}<N, \,0<n_{4}\leq n_{5}<N
                    \\ n_{3}\ge n_{4}}}
        \frac{1}{[n_{2}][n_{3}]}\frac{q^{N-n_{4}}}{[N-n_{4}][n_{5}]}
        \sum_{0<m_{1}}\left( \frac{q^{m_{1}+N-1-n_{2}}}{[m_{1}+N-1-n_{2}]}-\frac{q^{m_{1}+N-1}}{[m_{1}+N-1]} \right)
        \\
         & =\sum_{\substack{
                    0<n_{1}\leq n_{2}\leq n_{3}, \, 0<n_{4}\leq n_{5}<N
                    \\ n_{3}\ge n_{4}}}
        \frac{q^{N-n_{1}}}{[N-n_{1}][n_{2}][n_{3}]}\frac{q^{N-n_{4}}}{[N-n_{4}][n_{5}]}.
    \end{align}
    It is equal to the right-hand side of \eqref{eq:G-value} with $\bm{k}=(2,3)$.
\end{proof}

We now prove the main theorem of this section.

\begin{thm}
    For any index $\bm{k}$, it holds that
    \begin{align}
        F_{\bm{k}}(z)=G_{\overleftarrow{\bm{k}}}(z)
        \label{eq:q-F=G}
    \end{align}
    in the region $|z|<q^{-1}$.
\end{thm}

\begin{proof}
    Let $N$ be a positive integer.
    Note that replacing $q$ by $q^{-1}$ sends $[m]$ to
    $q^{1-m}[m]$ for any integer $m$.
        {}From Proposition \ref{prop:F-value}, we see that
    \begin{align*}
        q^{-|\bm{k}|}\, F_{\bm{k}}^{q}(q^{N-1})=
        \zeta_{<N}^{\star, q}(\bm{k})\bigg|_{q\to q^{-1}}.
    \end{align*}
    Similarly, from Proposition \ref{prop:G-value}, we find that
    \begin{align*}
        q^{-|\bm{k}|}\, G_{\overleftarrow{\bm{k}}}^{q}(q^{N-1})=
        \zeta_{<N}^{\star, q\flat}(\bm{k})\bigg|_{q\to q^{-1}}.
    \end{align*}
    Since the formula \eqref{eq:star-qMSW} is an identity of
    rational functions in $q$, it implies that
    $F_{\bm{k}}^{q}(q^{N-1})=G_{\overleftarrow{\bm{k}}}^{q}(q^{N-1})$
    for any positive integer $N$.
    Now the desired equality follows from the identity theorem
    for analytic functions.
\end{proof}

\begin{rem}
    In the case $\bm{k}=(1)$,
    the identity \eqref{eq:q-F=G} can be proved directly
    by means of the $q$-Gauss summation formula,
    as follows.

    We set $(z)_{m}=\prod_{j=0}^{m-1}(1-q^{j}z)$ for $m \ge 0$
    and $(z)_{\infty}=\prod_{j=0}^{\infty}(1-q^{j}z)$.
    It holds that
    \begin{align*}
        F_{(1)}^{q}(z)=-(1-q)\sum_{m\ge 1}
        \frac{1}{1-q^{m}}\frac{(z^{-1})_{m}}{(q)_{m}}
                         (qz)^{m}.
    \end{align*}
    Since
    \begin{align*}
        \frac{1}{1-q^{m}}=\lim_{a\to 1}\frac{1}{1-a}\frac{(a)_{m}}{(aq)_{m}}
    \end{align*}
    for $m\ge 1$,
    we find that
    \begin{align*}
        F_{(1)}^{q}(z)=(1-q)\lim_{a\to 1}\frac{1}{a-1}
        \left( {}_{2}\phi_{1}(a, z^{-1}, aq; qz)-1\right),
    \end{align*}
    where ${}_{2}\phi_{1}(a, b, c; z)$ is the
    $q$-hypergeometric series
    \begin{align*}
        {}_{2}\phi_{1}(a, b, c; z)=\sum_{m\ge 0}
        \frac{(a)_{m}(b)_{m}}{(c)_{m}(q)_{m}}z^{m}.
    \end{align*}
    The $q$-Gauss summation formula
    \begin{align*}
        {}_{2}\phi_{1}(a, b, c; c/ab)=
        \frac{(c/a)_{\infty}(c/b)_{\infty}}{(c)_{\infty}(c/ab)_{\infty}}
        \qquad (|c/ab|<1)
    \end{align*}
    implies that
    \begin{align*}
        F_{(1)}^{q}(z)=(1-q)\lim_{a\to 1}\frac{1}{a-1}
        \left( g(a; z)-1\right),
    \end{align*}
    where
    \begin{align*}
        g(a; z)=\frac{(q)_{\infty}(aqz)_{\infty}}{(aq)_{\infty}(qz)_{\infty}}.
    \end{align*}
    Since $g(1; z)=1$, we have
    \begin{align*}
        F_{(1)}^{q}(z)
         & =(1-q)\frac{\partial}{\partial a}g(a; z)\bigg|_{a=1}
        =(1-q)\frac{\partial}{\partial a}\left(\log{g(a; z)}\right)\bigg|_{a=1}
        \\
         & =(1-q)\sum_{m\ge 1}\left(\frac{q^{m}}{1-q^{m}}-\frac{q^{m}z}{1-q^{m}z}\right)=
        G_{(1)}^{q}(z).
    \end{align*}
\end{rem}

\end{document}